\documentclass[11pt,a4paper]{amsart}
\usepackage{amssymb}
\usepackage{amsthm}
\usepackage{amsmath}
\usepackage{amsxtra}
\usepackage{latexsym}
\usepackage{mathrsfs}
\usepackage[all,cmtip]{xy}
\usepackage[all]{xy}
\usepackage{enumitem}
\usepackage{xcolor}
\usepackage{graphicx}
\usepackage{comment}
\usepackage{mathabx,epsfig}
\usepackage{stmaryrd}
\def\acts{\ \rotatebox[origin=c]{-90}{$\circlearrowright$}\ }
\def\racts{\ \rotatebox[origin=c]{90}{$\circlearrowleft$}\ }

\def\uacts{\ \rotatebox[origin=c]{180}{$\circlearrowright$}\ }
\def\asubset{\ \rotatebox[origin=c]{-90}{$\subset$}\ }
\def\usubset{\ \rotatebox[origin=c]{90}{$\subset$}\ }

\usepackage{hyperref}

\newtheorem{thm}{Theorem}[section]
\newtheorem{lem}[thm]{Lemma}
\newtheorem{conj}[thm]{Conjecture}
\newtheorem{quest}[thm]{Question}

\newtheorem{prop}[thm]{Proposition}

\newtheorem{fact}[thm]{Fact}
\theoremstyle{definition}
\newtheorem{defn}[thm]{Definition}
\newtheorem{rmk}[thm]{Remark}
\newtheorem*{ack}{Acknowledgements}
\numberwithin{equation}{section}

\def\C{{\mathbb C}}
\def\Q{{\mathbb Q}}

\def\Z{{\mathbb Z}}
\def\P{{\mathbb P}}
\def\A{{\mathbb A}}

\def\F{{\mathbb F}}
\def\O{{ \mathcal{O}}}
\def\X{{ \mathcal{X}}}
\def\Y{{ \mathcal{Y}}}
\def\f{\varphi}
\def\d{\delta}
\def\lb{\llbracket}
\def\rb{\rrbracket}
\def\lp{(\!(}
\def\rp{)\!)}
\def\p{{\mathfrak{p}}}

\DeclareMathOperator{\pr}{pr}

\DeclareMathOperator{\id}{id}

\DeclareMathOperator{\Spec}{Spec}

\DeclareMathOperator{\Sing}{Sing}
\DeclareMathOperator{\Fix}{Fix}

  {\end{list}}

\title[]
{On Dynamical Cancellation}


\author{Jason P.~Bell}
\address{University of Waterloo \\
Department of Pure Mathematics \\
Waterloo, Ontario \\
Canada  N2L 3G1}

\email{jpbell@uwaterloo.ca}

\author{Yohsuke Matsuzawa}
\address{Department of Mathematics, Rikkyo University, 3-34-1 Nishi-Ikebukuro, Toshima-ku, Tokyo, 171-8501 JAPAN}
\email{matsuzaway@rikkyo.ac.jp}

\author{Matthew Satriano}
\address{University of Waterloo \\
Department of Pure Mathematics \\
Waterloo, Ontario \\
Canada  N2L 3G1}

\email{msatrian@uwaterloo.ca}

\thanks{J.B. and M.S. were partially supported by Discovery Grants from the National Science and Engineering Board of Canada.
Y.M. was partially supported by JSPS Overseas Research Fellowship.}

\begin{document}

\begin{abstract}
Let $X$ be a projective variety and let $f$ be a dominant endomorphism of $X$, both of which are defined over a number field $K$. We consider a question of the second author, Meng, Shibata, and Zhang, which asks whether the tower of $K$-points
$Y(K)\subseteq (f^{-1}(Y))(K)\subseteq (f^{-2}(Y))(K)\subseteq \cdots$ eventually stabilizes, where $Y\subset X$ is a subvariety invariant under $f$.  We show this question has an affirmative answer when the map $f$ is \'etale.  We also look at a related problem of showing that there is some integer $s_0$, depending only on $X$ and $K$, such that whenever $x, y \in X(K)$ have the property that $f^{s}(x) = f^{s}(y)$ for some $s \geq 0$, we necessarily have $f^{s_{0}}(x) = f^{s_{0}}(y)$. We prove this holds for \'etale morphisms of projective varieties, as well as self-morphisms of smooth projective curves. We also prove a more general cancellation theorem for polynomial maps on $\P^1$ where we allow for composition by multiple different maps $f_1,\dots,f_r$.
\end{abstract}

\maketitle

\setcounter{tocdepth}{2}
\tableofcontents 

\section{Introduction}
Let $X$ be a projective variety and let $f \colon X \rightarrow X$ be a surjective self-map, both of which are defined over a number field $K$.  A key component of understanding the dynamical system $(X,f)$ is to identify the invariant closed subschemes $Y$ of $X$; that is, $Y$ satisfying $f(Y)\subseteq Y$. In the case of such subschemes, the iterated preimages $f^{-n}(Y)$ of $Y$ under $f$ form a tower
\[
Y \subseteq f^{-1}(Y) \subseteq f^{-2}(Y) \subseteq \cdots.
\]
If $g: V\to V'$ is a surjective generically finite morphism of projective varieties then $g^*$ induces an injective homomorphism from the canonical ring of $V'$ into the canonical ring of $V$ and so the Kodaira dimension of $V$ is at least as large as that of $V'$ (see also the general conjecture of Iitaka \cite{ii}); in the case of curves we have a similar inequality involving genera coming from the Riemann-Hurwitz formula. These invariants serve as coarse proxies for the geometric complexity of varieties, and so one typically expects the irreducible components of $f^{-n}(Y)\smallsetminus f^{-(n-1)}(Y)$ to become increasingly geometrically complex as $n$ increases.

At the same time, there is a general principle within arithmetic geometry---now well-entrenched---that the geometric structure of a variety should exert significant influence over its underlying arithmetic structure.  The most famous manifestation of this principle is the celebrated theorem of Faltings \cite{fal1,fal2} (see also \cite{voj}), which asserts that an irreducible projective curve of genus at least two defined over a number field $K$ has only finitely many $K$-points.  More generally, there is a sweeping conjectural extension formulated by Bombieri and Lang (see \cite[Conjecture F.5.2.1]{hs}), which predicts that for an irreducible variety $X$ of general type (Kodaira dimension equal to its dimension as a variety) that is defined over a number field $K$, its $K$-points should lie in a proper closed subset. 

In light of the observation that the Kodaira dimensions of components in $f^{-n}(Y)$ are each at least as large as the Kodaira dimension of some component of $f^{-(n-1)}(Y)$, one expects that newly generated components under repeated backwards iteration should eventually be of general type, unless there is some compelling geometric reason for this not to occur. Consequently, we expect the tower of $K$-points
\begin{align*}
Y(K) \subseteq (f^{-1}(Y))(K) \subseteq (f^{-2}(Y))(K) \subseteq  \cdots \end{align*}
to eventually stabilize.

This expectation was made precise in the following question, originally asked in \cite[Question 8.4 (1)]{mmsz}.
\begin{quest}[Preimages Question]\label{PIC}
Let $f \colon X \longrightarrow X$ be a surjective self-morphism of a projective variety $X$ defined over a number field $K$.
Let $Y \subset X$ be a closed subscheme such that $f|_{Y}$ factors through $Y \subset X$.
Then there exists a non-negative integer $s_{0}$ such that 
\[
(f^{-(s+1)}(Y)\smallsetminus f^{-s}(Y))(K)=\emptyset
\]
for all $s\geq s_{0}$.
In other words, for $x \in X(K)$, if $f^{s}(x) \in Y(K)$ for some $s \geq 0$, then $f^{s_{0}}(x) \in Y(K)$.
\end{quest}

Strictly speaking, in \cite{mmsz}, the authors consider a stronger variant of this question involving points in $(f^{-(s+1)}(Y)\smallsetminus f^{-s}(Y))(L)$,
where $L$ varies over finite extensions of $K$ of bounded degree.  We shall also consider this stronger version of the question in this paper.  It should be noted that in the case of self-maps of $\mathbb{P}^N$, when $Y$ is a closed point, one can use Northcott's theorem and canonical heights (see \cite{CS}) to show the existence of an integer $s_0$ in the statement of Question \ref{PIC}.  Hence Question \ref{PIC} asks whether a type of uniform boundedness result should hold: can the integer $s_0$ can be uniformly bounded as one ranges over points in a closed invariant subscheme? 

Our first main result shows that the above question has an affirmative answer for \'etale self-maps. 

\begin{thm}
\label{thm:PIC-for-et-maps}
Let $X$, $Y$, $f$, and $K$ be as in the statement of Question \ref{PIC}.  Then the following hold:
\begin{enumerate}
\item[{\rm(}a{\rm )}] if $f$ is \'etale then there exists a non-negative integer $s_{0}$ such that 
\[
(f^{-(s+1)}(Y)\smallsetminus f^{-s}(Y))(K)=\emptyset
\]
for all $s\geq s_{0}$;

\item[{\rm(}b{\rm)}] if, in addition, $f|_Y$ is flat, then for each natural number $d$ there exists $s=s(X, Y, f,d)$ such that $
(f^{-(s+1)}(Y)\smallsetminus f^{-s}(Y))(L)=\emptyset$ for all field extensions $L$ of $K$ with $[L:K]\le d$.
\end{enumerate}
In particular, Question \ref{PIC} has an affirmative answer whenever $X$ is a smooth variety with nonnegative Kodaira dimension. \end{thm}

We note that Theorem \ref{thm:PIC-for-et-maps} (b) answers the original question asked in \cite[Question 8.4 (1)]{mmsz} under the additional hypotheses on the map $f$.


An important special case of the Preimages Question above arises when $g$ is a surjective self-map of a projective variety $Z$ defined over a number field $K$ and we consider the morphism $f:=(g,g)$ of $X:=Z\times Z$, along with the diagonal subvariety $Y=\{(x,x)\colon x\in X\}$.  In this case, Question \ref{PIC} has a particularly appealing formulation in terms of ``cancellation'': it asks whether there exists a natural number $s$ such that whenever $x,y\in X(K)$ are such that $f^n(x)=f^n(y)$ for some $n> s$ then we in fact have $f^s(x)=f^s(y)$.  

Even in this case when $X$ is $\mathbb{P}^1$, one quickly encounters thorny Diophantine problems concerning $K$-points on plane curves of the form $$(g(X)-g(Y))/(X-Y)=0.$$ 
In the polynomial case, the classification of which $g$ yield such curves with a genus zero component was worked out by Avanzi and Zannier \cite{AZ}; their work involves difficult estimates and on occasion requires a finer understanding of the genera of components of more general polynomial plane curves of the form $a(X)=b(Y)$, some of which relies on previous work of Bilu and Tichy \cite{BT}.  In the case of general rational functions, however, it is a notoriously difficult problem to classify rational functions $g(x)$ with the property that the curve $(g(X)-g(Y))/(X-Y)=0$ has a component of genus at most one.  In light of Faltings' theorem, this would be a key step in understanding which such curves have infinitely many $K$-points.  

We are nevertheless able to resolve the above cancellation question in the case of curves by applying $p$-adic uniformization techniques of Rivera-Letelier \cite{rl}.  
\begin{thm}
\label{thm:cancellation-for-curves}
Let $f \colon X \longrightarrow X$ be a surjective self-morphism of a projective curve $X$ defined over a number field $K$.
Then there exists a non-negative integer $s_{0}$ such that
for all $x, y \in X(K)$, if $f^{s}(x) = f^{s}(y)$ for some $s \geq 0$, then $f^{s_{0}}(x) = f^{s_{0}}(y)$.
\end{thm}

In light of Theorem \ref{thm:PIC-for-et-maps} and \ref{thm:cancellation-for-curves}, we expect that the following general result should hold.

\begin{conj}[Cancellation Conjecture]\label{cancelconj}
Let $f \colon X \longrightarrow X$ be a surjective self-morphism of a projective variety $X$ defined over a number field $K$.
Then there exists a non-negative integer $s_{0}$ satisfying the following property:
for all $x, y \in X(K)$, if $f^{s}(x) = f^{s}(y)$ for some $s \geq 0$, then $f^{s_{0}}(x) = f^{s_{0}}(y)$.
\end{conj}

We note that the assumption that $X$ be projective is necessary in the statement of Conjecture \ref{cancelconj} (and hence also in the statement of Question \ref{PIC}, if one wants an affirmative answer). Indeed, in Section \ref{subsec:counter-ex-non-proj}, we show that Conjecture \ref{cancelconj} fails for $X=\A^2$.  

A subtle issue is that, in general, Conjecture \ref{cancelconj} cannot be proved using $p$-adic methods, and hence one cannot expect to implement our proof strategy for Theorem \ref{thm:cancellation-for-curves} in higher dimensions. In particular, in Proposition \ref{prop:counter-ex-local-fields}, we construct a morphism $f\colon\P^2_\Q\to\P^2_\Q$ where Conjecture \ref{cancelconj} fails for $\Q_p$ for all primes $p\le \infty$.

In light of Theorem \ref{thm:PIC-for-et-maps}, the major content of Theorem \ref{thm:cancellation-for-curves} is that Conjecture \ref{cancelconj} holds for $\P^1$. As it was remarked earlier, significantly more is known in the case of polynomial self-maps of $\mathbb{P}^1$, due in large part to classical work of Ritt \cite{Ritt1}.  In particular, the behavior of semigroups of polynomial maps is much better understood than the case of general rational maps. This additional knowledge allows us to prove a more general cancellation result which allows for multiple different polynomial self-maps $\f_1,\dots,\f_r$ of $\mathbb{P}^1$.

\begin{thm}\label{thm:cancellationsemigroup}
Let $K$ be a number field and let $\f_{1}, \dots , \f_{r}$ be polynomial maps on $\P^{1}_{K}$ of degree at least two.
Then there is a proper closed subset $Z \subset \P^{1}_{K} \times_{K} \P^{1}_{K}$ with the following property:
for any points $a, b \in \P^{1}(K)$, if
\begin{align*}
\f_{i_{n}}\circ \cdots \circ \f_{i_{1}}(a) = \f_{i_{n}}\circ \cdots \circ \f_{i_{1}}(b)
\end{align*}
for some $n\geq 0$ and $i_{1}, \dots, i_{n} \in \{1,\dots , r\}$, 
then $(a,b) \in Z(K)$.
\end{thm}

From the viewpoint of cancellation conjecture, one might expect a statement of the form 
\[
\f_{i_{n}}\circ \cdots \circ \f_{i_{1}}(a) = \f_{i_{n}}\circ \cdots \circ \f_{i_{1}}(b) \Longrightarrow \f_{i_{n_{0}}}\circ \cdots \circ \f_{i_{1}}(a) = \f_{i_{n_{0}}}\circ \cdots \circ \f_{i_{1}}(b),
\]
but this does not hold in general. Theorem \ref{thm:cancellationsemigroup} is, in some sense, the optimal result one can expect for polynomial maps.
As a simple example, consider the case where $\f_1(x)=x^2$ and $\f_2(x)=x^3$. Then
\[
\f_1\circ \f_2^n(a)=\f_1\circ \f_2^n(-a), \quad  \f_2^n(a) \neq  \f_2^n(-a)
\]
for all $a\neq 0$ and all $n\geq0$.

\vspace{2em}
The outline of this paper is as follows.  In Section \ref{sec:PIC-etale}, we prove Theorem \ref{thm:PIC-for-et-maps}; we prove Theorem \ref{thm:cancellation-for-curves} in Section \ref{sec:cancel}.
In Section \ref{sec:genus}, we collect several important estimates involving the genera of curves; these will then be applied in Section \ref{sec:sgcancellation}, where we prove Theorem \ref{thm:cancellationsemigroup}.  Finally, in Section \ref{subsec:counter-ex-non-proj} we provide examples showing that Conjecture \ref{cancelconj} is false (and hence Question \ref{PIC} has a negative answer) if the projectivity assumption on $X$ is dropped or $K$ is replaced with a local field.

\begin{ack}
The second author would like to thank Joseph Silverman for discussing this subject with him and 
giving him many valuable suggestions.
He would also like to thank the department of mathematics at Brown University for hosting him during his 
fellowship.
\end{ack}

\section{The \'etale case}
\label{sec:etale}
In this section we prove Theorem \ref{thm:PIC-for-et-maps}.  The proof of part (a) is given in Subsection \ref{sec:PIC-etale}, and in Subsection \ref{subsec:etale-flat}, we prove part (b).
\vskip 2mm
\noindent{\bf Notation.} We adopt the following notation and assumptions for the remainder of this paper.
\begin{itemize}
\item A variety over a field $k$ is a separated irreducible reduced scheme of finite type over $k$.
\item A curve over a field $k$ is a one dimensional variety over $k$.
\item For a self-map $f \colon X \longrightarrow X$ on a set or variety $X$, the $n$-times composite of $f$ is denoted by $f^{n}$.
However, when, for example, $f \in \C_{p}\lb z \rb$ is a convergent power series, $f^{n}$ denotes the $n$-th power of $f$ as an element of the ring $\C_{p}\lb z \rb$.
The $n$-times composite of $f$ as a self-map on a small neighborhood of $0$, if it is defined, is denoted by $f^{\circ n}$.
\end{itemize}

We note that to prove Theorem \ref{thm:PIC-for-et-maps}, one may replace $f$ by $f^n$ and may also assume $f(Y)=Y$.

\subsection{Preimages Question for \'etale morphisms:~proof of Theorem \ref{thm:PIC-for-et-maps}}
\label{sec:PIC-etale}

Our goal in this subsection is to prove Theorem \ref{thm:PIC-for-et-maps}(a) and the concluding remark that Question \ref{PIC} has an affirmative answer for varieties of nonnegative Kodaira dimension. By \cite[Lemma 2.3]{fujimoto}, if $X$ is a smooth projective variety with $\kappa(X)\geq0$ and $f\colon X\to X$ is surjective, then $f$ is \'etale. So, we are reduced to showing the following result.

\begin{thm}\label{thm:picet1}
Question \ref{PIC} has an affirmative answer when $f$ is \'etale.
That is, let  $f \colon X \longrightarrow X$ be an \'etale self-morphism of a projective variety defined over a number field $K$.
Let $Y \subset X$ be a closed subscheme such that $f|_{Y}$ factors through $Y \subset X$.
Then there exists a non-negative integer $s_{0}$ such that 
\[
(f^{-s-1}(Y)\smallsetminus f^{-s}(Y))(K)=\emptyset
\]
for all $s\geq s_{0}$.
\end{thm}

The following result is the key ingredient needed in the proof of this theorem.

\begin{lem}\label{lem:spreadout}
Let $f \colon X \longrightarrow X$ be an \'etale self-morphism of a projective variety $X$ over a number field $K$.
Let $Y \subset X$ be a normal closed subscheme.
Then there is a non-negative integer $t$ such that for all $s \geq 0$, we have
\begin{align*}
\# \left\{ Z\  \middle|\  \txt{$Z$ is an irreducible component of $f^{-s}(Y)$ \\ such that $Z(K) \neq  \emptyset$}  \right\} \leq t.
\end{align*}
\end{lem}
\begin{proof}
Let $S$ be a sufficiently large finite set of primes of $K$ so that we have models $\X, F$, and $\Y$
over $\O_{K,S}$ of $X, f$, and $Y$ respectively with the following properties:
\begin{enumerate}
\item\label{model::1} $\X$ is an integral projective scheme over $\O_{K,S}$;
\item\label{model::2} $F \colon \X \longrightarrow \X$ is an \'etale surjective morphism over $\O_{K,S}$;
\item\label{model::3} $\Y$ is a closed subscheme of $\X$ and normal.
\end{enumerate}
It is clear that if we choose $S$ sufficiently large, we may assume (\ref{model::1}) and (\ref{model::2}) hold. Let us show that we may also assume (\ref{model::3}). Let $p \colon \Y \longrightarrow \Spec \O_{K,S}$ be any model of $Y$.
By enlarging $S$ if necessary, we may assume $\Y$ is a disjoint union of integral projective schemes over $\O_{K,S}$.
Since $Y \longrightarrow \Spec K$ is smooth outside a nowhere dense closed subset of $Y$,
we may assume $p$ is smooth outside a nowhere dense closed subset of $\Y$ by enlarging $S$ again if necessary.
In particular, the non-normal locus of $\Y$ is nowhere dense and this implies the normal locus of $\Y$ is open 
(cf.\ \cite[Chapter 12 (31.G) Lemma 3]{ma}). Since $Y$ is normal, the normal locus contains the generic fiber of $p$ and hence 
the image under $p$ of the non-normal locus of $\Y$ is a finite set. Adding this finite set to $S$, we obtain a normal model.

Now for any $s\geq 0$,  $F^{-s}(\Y)$ is a projective model of $f^{-s}(Y)$ over $\O_{K,S}$.
Moreover, since $F$ is \'etale and $\Y$ is normal, $F^{-s}(\Y)$ is a disjoint union of normal integral schemes:
\begin{align*}
F^{-s}(\Y) = \coprod_{i=1}^{r_s} \mathcal{Z}_{i}.
\end{align*}
Let $Z_{i}$ be the generic fiber of $ \mathcal{Z}_{i}$. Then the $Z_{i}$'s are the irreducible components of $f^{-s}(Y)$:
\begin{align*}
f^{-s}(Y) = \coprod_{i=1}^{r_s} Z_{i}.
\end{align*}
Since $ \mathcal{Z}_{i}$ is projective over $\O_{K,S}$, we have the canonical bijection $Z_{i}(K) \simeq \mathcal{Z}_{i}(\O_{K,S})$.
Let $\p$ be a maximal ideal of $\O_{K,S}$ and $k_{\p}$ be the residue field.
Then the reduction map 
\begin{align*}
\X(\O_{K,S}) \longrightarrow \X(k_{\p})
\end{align*}
maps $ \mathcal{Z}_{1}(\O_{K,S}), \dots , \mathcal{Z}_{r}(\O_{K,S})$ to disjoint subsets of $ \X(k_{\p})$.
Therefore, we must have
\begin{align*}
\# \{ i \mid Z_{i}(K) \neq  \emptyset \} = \# \{ i \mid \mathcal{Z}_{i}(\O_{K,S}) \neq  \emptyset \} \leq \# \X(k_{\p}).
\end{align*}
Since $\# \X(k_{\p})$ is finite and independent of $s$, we are done.
\end{proof}

\begin{proof}[Proof of Theorem \ref{thm:picet1}]
First of all, we may assume $Y$ is reduced.
If $Y$ is normal, then the statement easily follows from Lemma \ref{lem:spreadout}.

For the general case, we proceed by induction on $\dim Y$.
If $\dim Y=0$, then $Y$ is normal and we are done.

Suppose $\dim Y>0$.
Let us write $Y=Y_{1} \cup Y_{2}$ where
$Y_{1}$ is the union of all maximal dimensional irreducible components and
$Y_{2}$ is the union of all other components.
We equip $Y_{1}, Y_{2}$ with the reduced scheme structure. Then we have 
\begin{align*}
&f(Y_{1}) \subset Y_{1}\\
&f(\Sing Y_{1}) \subset \Sing Y_{1}
\end{align*}
where $\Sing Y_{1}$ is the non-regular locus of $Y_{1}$.
Indeed, the first statement is clear by the definition of $Y_{1}$ and the fact that $f$ is finite.
For the second statement, note the composition of $f|_{Y_{1}}$ with the inclusion map $Y_1\to X$ is equal to $Y_{1} \longrightarrow X \stackrel{f}{\longrightarrow} X$, which is unramified since $f$ is \'etale.
Hence $f|_{Y_{1}} \colon Y_{1}  \longrightarrow Y_{1}$ is also unramified.
Since $Y_{1}$ is pure dimensional, we see $f(\Sing Y_{1}) \subset \Sing Y_{1}$.

Let $\nu \colon \widetilde{Y_{1}} \longrightarrow Y_{1}$ be the normalization, i.e.
$ \widetilde{Y_{1}}$ is the union of the normalization of every irreducible component of $Y_{1}$.
Take a closed immersion $i \colon \widetilde{Y_{1}} \longrightarrow \P^{N}_{K}$ and
let us view $ \widetilde{Y_{1}}$ as a closed subscheme of $\P^{N} \times X$ via
\[
\widetilde{Y_{1}} \xrightarrow{(i, \nu)} \P^{N} \times Y_{1} \subset \P^{N} \times X.
\]
Set $g=\id \times f \colon \P^{N}\times X \longrightarrow \P^{N}\times X$.
Since $g$ is \'etale and $ \widetilde{Y_{1}}$ is normal, by Lemma \ref{lem:spreadout},
there exists a positive integer $t$ such for all $s$, at most $t$ irreducible components of $g^{-s}(\widetilde{Y_{1}})$
have $K$-points. 

Note that the diagram 
\[
\xymatrix{
g^{-s-1}(\widetilde{Y_{1}}) \ar[r]^(.65){g^{s+1}} \ar[d]& \widetilde{Y_{1}} \ar[d]^{\nu}\\
f^{-s-1}(Y_{1}) \ar[r]_(.65){f^{s+1}} & Y_{1}
}
\]
is cartesian and $\nu$ is an isomorphism over $Y_{1} \smallsetminus \Sing Y_{1}$.
Therefore, there exists a positive integer $s_{0}$ such that 
\[
(f^{-s-1}(Y_{1}) \smallsetminus f^{-s}(Y_{1}))(K) \subset (f^{-s-1}(\Sing Y_{1}))(K)
\]
for all $s\geq s_{0}$.

Let $Y_{3}\subset Y_{2}$ be the union of all irreducible components which are eventually mapped into $Y_{1}$ and
let $Y_{4} \subset Y_{2}$ be the union of all other components.
Then we have $f(Y_{4}) \subset Y_{4}$.
Take $r>0$ so that $f^{r}(Y_{3})\subset Y_{1}$.
Then for all $s\geq0$, we have 
\[
(f^{-s-1}(Y_{3})\smallsetminus f^{-s}(Y))(K) \subset (f^{-s-1-r}(Y_{1})\smallsetminus f^{-s}(Y_{1}))(K).
\]
Then for $s \geq s_{0}$,
\begin{align*}
&(f^{-s-1}(Y)\smallsetminus f^{-s}(Y))(K) \\
&\subset (f^{-s-1}(Y_{1})\smallsetminus f^{-s}(Y_{1}))(K) \cup (f^{-s-1}(Y_{3})\smallsetminus f^{-s}(Y))(K) \\
& \qquad \qquad \qquad \qquad \qquad \qquad \qquad \qquad \cup (f^{-s-1}(Y_{4})\smallsetminus f^{-s}(Y_{4}))(K)\\
&\subset (f^{-s-1-r}(Y_{1})\smallsetminus f^{-s}(Y_{1}))(K) \cup (f^{-s-1}(Y_{4})\smallsetminus f^{-s}(Y_{4}))(K)\\
&\subset (f^{-s-1-r}(\Sing Y_{1})\smallsetminus f^{-s}(\Sing Y_{1}))(K) \cup (f^{-s-1}(Y_{4})\smallsetminus f^{-s}(Y_{4}))(K).
\end{align*}
Applying the induction hypothesis to $X, (\Sing Y_{1})\cup Y_{4}$, and $f$, we are done.
\end{proof}

\subsection{Allowing for bounded extensions of $K$}
\label{subsec:etale-flat}

When $f$ is \'etale and the induced morphism on $Y$ is flat, 
we can prove a stronger statement, as given by Theorem \ref{thm:PIC-for-et-maps} (b).

\begin{thm}\label{thm:picet2}
Let $f \colon X \longrightarrow X$ be a surjective self-morphism of a projective variety $X$ defined over a number field $K$.
Let $Y \subset X$  be a closed subscheme such that $f(Y)\subset Y$ {\rm (}i.e.\ $f|_{Y}$ factors through $Y \subset X${\rm )}.
Suppose 
\begin{itemize}
\item $f \colon X \longrightarrow X$ is \'etale;
\item $f|_{Y} \colon Y \longrightarrow Y$ is flat.
\end{itemize}
Fix a positive integer $d$.
Then there exists a positive integer $r$ satisfying the following:
for any finite extension $L$ of $K$ with $[L:K]\leq d$, and
for any point $x \in X(L)$ such that $f^{n}(x)\in Y(L)$ for some $n\geq 0$, we have $f^{r}(x)\in Y(L)$.
\end{thm}

\begin{rmk}
The proof of Theorem \ref{thm:picet2} shows there is an effectively computable upper bound on $r$ depending only on $d$ and the integral models of
$X, Y$, and $f$.
\end{rmk}

\subsubsection{Over local fields}

Theorem \ref{thm:picet2} can be completely reduced to the similar statement over a local field.
Fix a prime $p$ and let $K$ be a finite extension of $\Q_{p}$, $\O \subset K$ be its ring of integers,
and $\F$ be the residue field of $\O$, which is a finite field.
 
\begin{prop}\label{localver1}
Let $\X$ be a projective scheme over $\Spec \O$ and let $f \colon \X \longrightarrow \X$ be a finite surjective morphism over $\Spec \O$.
Let $\Y \subset \X$ be a closed subscheme such that $f(\Y)\subset \Y$ 
{\rm (}i.e.\ $f|_{\Y} \colon \Y \longrightarrow \X$ factors through $\Y${\rm )}.
Suppose 
\begin{itemize}
\item $f \colon \X \longrightarrow \X$ is \'etale;
\item $f|_{\Y} \colon \Y \longrightarrow \Y$ is flat.
\end{itemize}
Let $X=\X \times_{\O}K$, $Y=\Y \times_{\O}K$, and let $f_{K}$ be the induced morphism on $X$.
Let $x\in X(K)$ be a point such that $f_{K}^{n}(x) \in Y(K)$ for some $n\geq0$.
Then $f_{K}^{r}(x) \in Y(K)$ for all $r\geq \# \X(\F)$.
\end{prop}
\begin{proof}
We canonically identify $X(K)$ and $\X(\O)$ (similarly $Y(K)$ and $\Y(\O)$).
We consider $\Y(\O)$ as a subset of $\X(\O)$.
Let $ \overline{x}$ denote the reduction of $x\in X(K)=\X(\O)$ and let $ \overline{f}$ be the reduction of $f$.
Let $a$ (respectively $b$) be the preperiod (respectively period) of $ \overline{x}$ under $ \overline{f}$.
Note that $a+b\leq \# \X(\F)$.

Set $y=f^{a}(x)$ and $g=f^{b}$.
Then we have $g(\Y) \subset \Y$, $ \overline{g}( \overline{y})= \overline{y}$, and $g^{m}(y)\in \Y(\O)$ for some $m>0$.
Note that $ \overline{g}^{m}( \overline{y})= \overline{y}\in \Y(\F)$.
We will show that $y\in \Y(\O)$.

We have the following commutative diagram
\begin{equation}\label{maindiagram-et-flat}
\begin{gathered}
\xymatrix{
\Spec \O \ar@/_18pt/[dddr]_{y} \ar@/^18pt/[rrd]^{\id}& &\\
               &\Spec \F \ar[d]_{ \overline{y}} \ar[lu]^-{\iota} \ar[r]^-{\iota} & \Spec \O \ar[d] \ar@/^18pt/[dd]^{g^{m}(y)}\\
               &\Y \ar[d] \ar[r]^{h} & \Y \ar[d]\\
               &\X \ar[r]_{g^{m}}& \X
}
\end{gathered}
\end{equation}
where $\iota$ denotes the inclusion of the closed point and, for ease of notation, we let $h=g^m|_\Y$.

Let $S=\X\times_{g^m,\X,g^m(y)}\Spec\O$, from which we obtain a commutative diagram
\[
\xymatrix{
\Spec\O\ar[r]^-{j}\ar[dr]_-{y}\ar@/^18pt/[rr]^-{\id} & S\ar[r]\ar[d] &\Spec\O  \ar[d]^{g^{m}(y)}\\
& \X \ar[r]_{g^{m}} & \X 
}
\]
where the square is cartesian. Since $g^m$ is \'etale, $S$ is the disjoint union of spectra of DVRs. Next, let $T=\Y\times_{h,\Y,g^m(y)}\Spec\O$, from which we obtain a commutative diagram
\[
\xymatrix{
\Spec\F\ar[r]^-{i}\ar[dr]_-{\overline{y}}\ar@/^18pt/[rr]^-{\iota} & T\ar[r]\ar[d] &\Spec\O  \ar[d]^{g^{m}(y)}\\
& \Y \ar[r]_{h} & \Y 
}
\]
where the square is cartesian. Since $h$ is flat, $T$ is the disjoint union of spectra of one-dimensional local rings.

From diagram \eqref{maindiagram-et-flat}, we obtain a commutative diagram
\begin{equation}\label{secondmaindiagram-et-flat}
\begin{gathered}
\xymatrix{
\Spec\F \ar[r]^-{i}\ar[d]^-{\iota} &T  \ar[d]^-{\alpha}\\
\Spec\O \ar[r]^-{j} & S
}
\end{gathered}
\end{equation}
where $\alpha$ is a closed immersion. Given our above descriptions of $S$ and $T$ in terms of spectra of local rings, it follows that $j$ and $\alpha$ are open immersions.  By diagram \eqref{secondmaindiagram-et-flat}, we see $j$ maps the closed point of $\Spec\O$ to $T$, and hence $j$ lifts to $Y$, that is, $y\in\Y(\O)$.
\end{proof}

\begin{prop}\label{localver2}
Keep the notation of Proposition \ref{localver1}. Let $\F = \F_{q}$, the finite field of $q$ elements.
For any $d>0$, let $r = \#\X(\F_{q^{d!}})$.
Then for any finite extension $L$ of $K$ with $[L:K]\leq d$ and for any point $x\in X(L)$ such that $f_{L}^{n}(x)\in Y(L)$ for some $n\geq0$,
we have $f_{L}^{r}(x)\in Y(L)$.
Here $f_{L}$ is the base change of $f_{K}$.
\end{prop}

\begin{proof}
Let $\O'$ be the integral closure of $\O$ in $L$ and $\F'$ be the residue field of $\O'$.
Since $[L:K]\leq d$, we have $[\F' : \F_{q}]\leq d$.
Then $$\#(\X\times_{\O}\O') (\F')=\#\X(\F')\leq \#\X(\F_{q^{d!}})=r.$$
By applying Proposition \ref{localver1} to $f\times_{\O}\O' \colon \X\times_{\O}\O' \longrightarrow \X\times_{\O}\O'$ and $\Y\times_{\O}\O'$,
we get $f_{L}^{r}(x)\in Y(L)$.
\end{proof}

\subsubsection{Proof of Theorem \ref{thm:picet2}}

Theorem \ref{thm:picet2} now easily follows from Proposition \ref{localver2}.

\begin{proof}[Proof of Theorem \ref{thm:picet2}]
Take a model $X_{\O_{K, S}}, Y_{\O_{K, S}}$, and $f_{\O_{K, S}}$ of $X, Y$, and $f$ over the ring of $S$-integers $\O_{K, S}$ 
where $S$ is a sufficiently large finite set of places of $K$ (depending only on $X$ and $f$) so that
\begin{enumerate}
\item $Y_{\O_{K, S}} \subset X_{\O_{K, S}}$ is a closed subscheme and $f_{\O_{K, S}}|_{Y_{\O_{K, S}}}$
factors through $Y_{\O_{K, S}} \subset X_{\O_{K, S}}$;
\item $f_{\O_{K,S}}$ is \'etale and $f_{\O_{K,S}}|_{Y_{\O_{K,S}}}$ is flat.
\end{enumerate}

Pick a finite place $v \notin S$ of $K$ and
let $\X, \Y$, and $f_{v}$ be the base change to the completion $\O_{v}$ of $\O_{K,S}$ at $v$.
Then for any point $x \in X(L)$ where $L$ is a finite extension of $K$ with $[L:K]\leq d$, 
if we choose a place $w$ of $L$ lying over $v$, then $x$ can be seen as a point of $X(L_{w}) = (\X \times_{\O_{v}} K_{v})(L_{w})$.
Note that $[L_{w}:K_{v}] \leq [L:K] \leq d$.
Applying Proposition \ref{localver2} to $\X, \Y$, and $f_{v}$, the theorem is proved.
\end{proof}

\section{Cancellation on curves}
\label{sec:cancel}
The goal of this section is to prove Theorem \ref{thm:cancellation-for-curves}. By replacing our curve $X$ by its normalization, we may assume that the ambient curve is smooth. By Theorem \ref{thm:PIC-for-et-maps}, we may also assume that are curve has Kodaira dimension $-\infty$ and hence is $\mathbb{P}^1$, and so we see it suffices to prove the following result.  

\begin{thm}\label{thm:cancellatioonP1}
Let $K$ be a number field.
Let $f \colon \P^{1}_{K} \longrightarrow \P^{1}_{K}$ be a surjective morphism of degree $\deg f \geq 2$.
Then there is a non-negative integer $N$ such that if $a, b \in \P^{1}(K)$ satisfy $f^{n}(a)=f^{n}(b)$ for some $n \geq 0$,
then $f^{N}(a) = f^{N}(b)$.
\end{thm}

\subsection{$p$-adic unifomization}

The main ingredient of the proof of cancellation on $\P^{1}$ is 
the $p$-adic uniformization theorem due to Rivera-Letelier \cite{rl}.
We summarize some basic notions and theorems on power series over $\C_{p}$ that we will use.\\
\\
\noindent {\bf Notation.} Throughout this section, we make use of the following notation.

\begin{enumerate}
\item For $r>0$, 
\[
\C_{p}\{ r^{-1}z \} = \{ f=\sum_{i \geq 0} a_{i} z^{i} \in \C_{p}\lb z \rb \mid \text{$|a_{i}|r^{i} \to 0$ as $i \to \infty$}\}.
\]
This is a Banach algebra over $\C_{p}$ with multiplicative norm $\| f\|_{r} = \max_{i} |a_{i}|r^{i}$.

\item $\O = \bigcup_{r>0} \C_{p}\{ r^{-1}z \}$

\item $D(a, r) = \{ z\in \C_{p} \mid |z-a| < r \}$.

\end{enumerate}

We are interested in the local behavior of convergent power series $f \in \O$ considered as 
a map between subsets of $\C_{p}$.
The next lemma says such maps are contractive under a suitable condition.

\begin{lem}[cf.\ {\cite[Proposition 3.2]{rl}}]\label{lem:contraction}
Let $f = a_{1}z + a_{2}z^{2} + \cdots \in \O$ with $|a_{1}| < 1$.
Let $r(f) = \bigl( \sup_{k\geq 1} |a_{k}|^{1/k} \bigr)^{-1}$.
Then for $z \in \C_{p}$ with $|z| < r(f)$, we have
\[
|f(z)| \leq |z| \max\left\{ |a_{1}|, \frac{|z|}{r(f)}  \right\}.
\]
\end{lem}

The following theorem is so called $p$-adic uniformization theorem, which is due to Rivera-Letelier \cite{rl}.

\begin{thm}[cf.\ {\cite[Proposition 3.3]{rl}}]\label{thm:padicunif}
Let $f = a_{1}z + a_{2}z^{2} + \cdots \in \O$ with $|a_{1}| < 1$.
Let $r(f) = \bigl( \sup_{k\geq 1} |a_{k}|^{1/k} \bigr)^{-1}$.
Then for all $r < r(f)$, $f^{\circ n} \in \C_{p}\{ r^{-1}z\}$ and $\| f^{\circ n} \|_{r} < r(f)$.

\begin{enumerate}
\item
Suppose $a_{1} \neq 0$.
Then there is a power series $u \in \C_{p}\lb z \rb$ such that 
\begin{itemize}
\item $u \in \C_{p}\{r^{-1}z\}$ for all $r< r(f)$;
\item the sequence $\{ a_{1}^{-n} f^{\circ n}\}$ converges to $u$ in $\C_{p}\{r^{-1}z\}$ for any $r < r(f)$;
\item $u \circ f = a_{1}u$
\end{itemize}
Moreover, $u$ has a composition inverse in $\O$ and defines a bijection between small open balls around $0$.
If $f \in K\lb z \rb$ for some complete subfield $K \subset \C_{p}$, then $u \in K\lb z \rb$.

\item
Suppose $a_{1} = 0$ and let $f = a_{d}z^{d} + a_{d+1}z^{d+1} + \cdots$ with $a_{d} \neq 0$.
Then there is $u \in \O$ such that
\[
u \circ f = u^{d}.
\]
Moreover, $u$ has a composition inverse in $\O$  and defines a bijection between small open balls around $0$.
If $f \in K\lb z \rb$ for some finite extension $K$ of $\Q_{p}$, then we can take $u \in K(b)\lb z \rb$ where
$b$ is a fixed $(d-1)$-st root of $a_{d}$.

\end{enumerate}

\end{thm}

The above two facts are about the behavior of $f$ around a fixed point, namely $0$.
For a general $f$, $0$ is not necessarily a fixed point, but the following lemma provides a fixed point under a suitable condition.

\begin{lem}[cf.\ {\cite[Lemma 6.2.1.1]{bgt}}]\label{lem:fixedpt}
Let $f = a_{0} + a_{1}z + a_{2}z^{2} + \cdots \in \Z_{p}\lb z \rb$ with $|a_{0}|, |a_{1}| < 1$.
Then the map $f \colon p\Z_{p} \longrightarrow p\Z_{p}$ has a fixed point.
\end{lem}

\subsection{Cancellation on $\P^{1}$:~proof of Theorem \ref{thm:cancellation-for-curves}}

Now we prove the Cancellation Conjecture for $\P^{1}$. 
We use the following two facts.
The first one will be used to reduce the problem to one over $\Z_{p}$;
the second one will be used in the final step of the proof.

\begin{prop}[cf.\ {\cite[Proposition 2.5.3.1]{bgt}}]\label{fact:embQp}
Let $K$ be a finitely generated field over $\Q$ and $S \subset K$ a finite set.
Then there are infinitely many prime numbers $p$ such that there exists a field embedding $\sigma \colon K \longrightarrow \Q_{p}$ satisfying 
$\sigma(S) \subset \Z_{p}$.
\end{prop}

\begin{fact}[cf.\ {\cite[II (7.12), (7.13)]{neu}}]\label{fact:rootofunity}
Let $p$ be a prime number.
\begin{enumerate}
\item
Let $m\in \Z_{>0}$ and $\zeta \in \overline{\Q_{p}}$ be a primitive $p^{m}$-th root of unity.
Then $[\Q_{p}(\zeta) : \Q_{p}]=(p-1)p^{m-1}$.
\item
Let $n \in \Z_{>0}$ such that $p \nmid n$.
Let $\zeta \in \overline{\Q_{p}}$ be a primitive $n$-th root of unity.
Then $[\Q_{p}(\zeta) : \Q_{p}] = l$, where $l$ is the order of $p$ in the unit group of $\Z/n\Z$. 
In particular, $p^{l} > n$, i.e. 
\[
\frac{\log n}{\log p} < [\Q_{p}(\zeta) : \Q_{p}].
\]
\item\label{fact:rootofunity::3}
Let $D \geq 1$ be an integer.
Then the set
\[
\{ \zeta \in \overline{\Q_{p}} \mid \text{$\zeta$ is a root of unity such that $[\Q_{p}(\zeta) : \Q_{p}] \leq D$}\}
\]
is finite.
\end{enumerate}

\end{fact}

\noindent {\bf Convention.}
For a map $f \colon X \longrightarrow X$ on a set $X$, we say that ``the cancellation statement holds'' (or ``cancellation holds'') if
there is a non-negative integer $N$ such that for any points $a, b \in X$ satisfying $f^{n}(a) = f^{n}(b)$ for some $n \geq 0$,
we have $f^{N}(a) = f^{N}(b)$.

\begin{proof}[Proof of Theorem \ref{thm:cancellatioonP1}]
We prove the result through several steps.\\
\\
\noindent{\bf Step 1.} Let $f \colon \P^{1}_{K} \longrightarrow \P^{1}_{K}$ be a surjective morphism with $\deg f \geq 2$.
Choose a model
 \[
\xymatrix{
f \acts \P^{1}_{K} \ar[r] \ar@<1.5ex>[d] & \P^{1}_{R} \racts f_{R} \ar@<-1.5ex>[d] \\
\Spec K  \ar[r] &  \Spec R
}
\]
where $R$ is the localization of the ring of integers $\O_{K}$ by an element and $f_{R}$ is a finite surjective self-morphism on $\P^{1}_{R}$ over $R$
such that $f$ is the base change of $f_{R}$.
Using Proposition \ref{fact:embQp}, we may pick a prime number $p$ and embedding $\sigma \colon K \longrightarrow \Q_{p}$ such that $\sigma(R) \subset \Z_{p}$.
Then we have the following diagram:
 \[
\xymatrix@R=-1mm{
&f_{R}&\\
&\uacts&\\
f \acts \P^{1}_{K} \ar[r] \ar@<1.5ex>[dddddd] & \P^{1}_{R}  \ar[dddddd] & \P^{1}_{\Z_{p}} \racts f_{\Z_{p}} \ar@<-2ex>[dddddd] \ar[l]\\
&&\\
&&\\
&&\\
&&\\
&&\\
\Spec K  \ar[r] &  \Spec R & \Spec \Z_{p} \ar[l]
}
\]
where $f_{\Z_{p}}$ is the base change of $f_{R}$.
Note that $\P^{1}(K)$ is canonically embedded in $\P^{1}(\Z_{p})$.
We will actually prove cancellation for points in $\P^{1}(\Z_{p})$ and the morphism $f_{\Z_{p}}$,
which is obviously enough to show the theorem.
In the rest of the proof, we write $f$ and $\P^1$ in place of $f_{\Z_{p}}$ and $\P^1_{\Z_{p}}$.\\
\\
\noindent{\bf Step 2.} Let $f \colon \P^{1}_{\Z_{p}} \longrightarrow \P^{1}_{\Z_{p}}$ be a finite surjective morphism over $\Z_{p}$.
We will prove the following:
\[
\txt{There is a non-negative integer $N$ such that if $a, b \in \P^{1}(\Z_{p})$ \\satisfies $f^{n}(a)=f^{n}(b)$ for some $n \geq 0$,
then $f^{N}(a) = f^{N}(b)$. }
\]

Let $ \overline{f} \colon \P^{1}_{\F_{p}} \longrightarrow \P^{1}_{\F_{p}}$ denote the reduction of $f$ modulo $p$, and for any $a\in\P^1(\Z_p)$, let $\overline{a}\in\P^1(\F_p)$ denote the reduction. 
Since there are only finitely many self-maps of $\P^{1}(\F_{p})$, we can pick a positive integer $N_{0}$ such that
\begin{align*}
\overline{f}^{N_{0}}|_{\P^{1}(\F_{p})} = \overline{f}^{2N_{0}}|_{\P^{1}(\F_{p})}.
\end{align*}
Set $g = f^{N_{0}}$.
Then: 
\begin{enumerate}
\item\label{midpf::1} for any $a \in \P^{1}(\Z_{p})$, if we set $c = f^{N_{0}}(a)$, then $ \overline{g}( \overline{c}) = \overline{c}$;
\item\label{midpf::2} for any $a, b \in \P^{1}(\Z_{p})$ such that $f^{n}(a) = f^{n}(b)$ for some $n \geq 0$, we have $ \overline{f^{N_{0}}(a)} = \overline{f^{N_{0}}(b)}$.
\end{enumerate}
Indeed, we have
\begin{align*}
\overline{g}( \overline{c}) = \overline{f^{N_{0}}}( \overline{f^{N_{0}}(a)}) = \overline{f}^{2N_{0}}( \overline{a}) = \overline{f}^{N_{0}}( \overline{a}) = \overline{c}.
\end{align*}
For the second statement, note that $g^{m}(f^{N_{0}}(a)) = g^{m}(f^{N_{0}}(b))$ for $m$ sufficiently large. 
Then we get $ \overline{g}^{m}( \overline{f^{N_{0}}(a)}) = \overline{g}^{m}( \overline{f^{N_{0}}(b)})$ and by the first statement, 
we get $ \overline{f^{N_{0}}(a)} = \overline{f^{N_{0}}(b)}$.

Next, for any $\xi \in \P^{1}(\F_{p})$ such that $ \overline{g}(\xi) = \xi$ (i.e. for $\xi \in \Fix( \overline{g}|_{\P^{1}(\F_{p})})$),
$g$ induces a $\Z_{p}$-algebra homomorphism
 \[
\xymatrix{
\Z_{p}\lb z \rb \ar@{}[r]|{\simeq} & \widehat{\O}_{\P^{1}_{\Z_{p}},\xi} \ar[r]^{g^{*}} & \widehat{\O}_{\P^{1}_{\Z_{p}},\xi} \ar@{}[r]|{\simeq}& \Z_{p}\lb z \rb \\
z \ar@{|->}[rrr] &&& G_{\xi}(z)
}
\]
where $G_{\xi}(z)\in\Z_{p}\lb z\rb\smallsetminus\Z_{p}$ whose constant term is equal to $0$ modulo $p$.

Set
\[
U_{\xi} = \{ \alpha \in \P^{1}(\Z_{p}) \mid \overline{ \alpha}=\xi\}.
\]
Then $g(U_{\xi}) \subset U_{\xi}$ since $ \overline{g}(\xi)= \xi$ and
 \[
\xymatrix{
U_{\xi} \ar@{}[r]|{\simeq} \ar[d]_{g}& p\Z_{p} \ar[d]^{G_{\xi}}\\
U_{\xi} \ar@{}[r]|{\simeq} & p\Z_{p}.
}
\]

Now, for $a, b \in \P^{1}(\Z_{p})$ with $f^{n}(a) = f^{n}(b)$ for some $n\geq0$, 
by (\ref{midpf::1}) and (\ref{midpf::2}) above, $ \overline{f^{N_{0}}(a)} = \overline{f^{N_{0}}(b)}$ is a fixed point of $ \overline{g}$.
Therefore, to end the proof, we may assume that $a, b$ are both contained in $U_{\xi}$ for some $\xi \in \Fix( \overline{g}|_{\P^{1}(\F_{p})})$.
In particular, it is enough to show the cancellation statement for $g \colon U_{\xi} \longrightarrow U_{\xi}$ for each such $\xi$.\\
\\
\noindent{\bf Step 3.} Let us fix a $\xi \in \Fix( \overline{g}|_{\P^{1}(\F_{p})})$.
Let us write 
\[
G(z) = G_{\xi}(z) = a_{0}+ a_{1}z + a_{2} z^{2} +\cdots.
\]
Let $|\ |$ denote the $p$-adic absolute value. Then $|a_{0}| < 1$.

Suppose first that $|a_{1}|=1$. In this case, the power series $a_{1}z + a_{2} z^{2} +\cdots$ has a composition inverse in $z\Z_{p}\lb z \rb$.
Thus $G \colon p\Z_{p} \longrightarrow p\Z_{p}$ is bijective, hence so is $g \colon U_{\xi} \longrightarrow U_{\xi}$.
Thus cancellation holds for $g \colon U_{\xi} \longrightarrow U_{\xi}$.

So, we may suppose $|a_{1}| < 1$. Then by Lemma \ref{lem:fixedpt}, there is $z_{0} \in p\Z_{p}$ such that $G(z_{0})=z_{0}$.
Then we can write $G(z) = z_{0} + c_{1}(z-z_{0}) + c_{2}(z-z_{0})^{2} + \cdots$ where $c_{i} \in \Z_{p}$.
Then we have the following commutative diagram:
 \[
\xymatrix{
U_{\xi} \ar@{}[r]|{\simeq} \ar[d]_{g} & p\Z_{p} \ar[d]_{G} & p\Z_{p} \ar[l]_{+z_{0}} \ar[d]^{ \widetilde{G}} \\
U_{\xi} \ar@{}[r]|{\simeq} & p\Z_{p} & p\Z_{p} \ar[l]_{+z_{0}} 
}
\]
where $ \widetilde{G}(z) = c_{1}z + c_{2}z^{2}+\cdots$.
It is enough to prove cancellation for $ \widetilde{G} \colon p\Z_{p} \longrightarrow p\Z_{p}$. We treat the two cases, $c_{1}\neq 0$ and $c_{1}=0$ separately.\\
\\
\noindent\underline{Case 1}.
Suppose $c_{1} \neq 0$.
Then by Theorem \ref{thm:padicunif} (1), there are $r, s>0$ and $u \in \Z_{p}\lb z \rb$
such that 
\begin{itemize}
\item $u$ converges on $D(0,r)$ and defines a bijection $u \colon D(0,r) \longrightarrow D(0,s)$; 

\vspace{1mm}
\item $ u \circ \widetilde{G} = c_{1}u$ as elements in $\Z_{p}\lb z \rb$.
\end{itemize}
We may take $r < 1$.
Then we have a commutative diagram
 \[
\xymatrix{
D(0,r) \ar[r]^{u} \ar[d]_{ \widetilde{G}} & D(0,s) \ar[d]^{c_{1}\cdot} \\
D(0,r) \ar[r]_{u} & D(0,s)
}
\]
and by Lemma \ref{lem:contraction}, there is an integer $N \geq 0$ such that $ \widetilde{G}^{\circ N}(p\Z_{p}) \subset D(0,r)$.

Now suppose $a, b \in p\Z_{p}$ satisfy $ \widetilde{G}^{\circ n}(a) = \widetilde{G}^{\circ n}(b)$ for some $n \geq 0$. Then
$ \widetilde{G}^{\circ m}( \widetilde{G}^{\circ N}(a)) = \widetilde{G}^{\circ m}( \widetilde{G}^{\circ N}(b))$ for some $m\geq 0$.
So, we have 
\[
 c_{1}^{m} u(\widetilde{G}^{\circ N}(a))=c_{1}^{m} u(\widetilde{G}^{\circ N}(b))
\]
and this implies $\widetilde{G}^{\circ N}(a) = \widetilde{G}^{\circ N}(b)$.\\
\\
\noindent\underline{Case 2}. Suppose $c_{1}=0$.
Let us write $\widetilde{G}(z) = c_{d}z^{d}+c_{d+1}z^{d+1}+\cdots$ where
$d \geq 2$ and $c_{d} \neq 0$.
Fix an $(d-1)$-st root $c_{d}^{1/(d-1)}$ of $c_{d}$ in $ \overline{\Q_{p}}$ and
set $L = \Q_{p}(c_{d}^{1/(d-1)})$.
By Theorem \ref{thm:padicunif} (2), 
there are $0< r, s <1$ and $u \in L\lb z \rb$ such that
\begin{itemize}
\item $u$ converges on $D(0,r)$ and defines a bijection $u \colon D(0,r) \longrightarrow D(0,s)$; 
\vspace{1mm}
\item $u \circ \widetilde{G} = u^{d}$ as elements in $L\lb z \rb$.
\end{itemize}

Then the diagram
 \[
\xymatrix{
D(0,r) \ar[r]^{u} \ar[d]_{ \widetilde{G}} & D(0,s) \ar[d]^{z^{d}} \\
D(0,r) \ar[r]_{u} & D(0,s)
}
\]
commutes and by Lemma \ref{lem:contraction}, there is an integer $l \geq 0$ such that $ \widetilde{G}^{\circ l}(p\Z_{p}) \subset D(0,r)$.

Now suppose $a, b \in p\Z_{p}$ satisfies $ \widetilde{G}^{\circ n}(a) = \widetilde{G}^{\circ n}(b)$ for some $n \geq 0$, then
$\widetilde{G}^{\circ m}( \widetilde{G}^{\circ l}(a)) =\widetilde{G}^{\circ m}( \widetilde{G}^{\circ l}(b))$ for large $m \geq 0$.
This implies
\[
(u(\widetilde{G}^{\circ l}(a)))^{d^{m}} =(u(\widetilde{G}^{\circ l}(b)))^{d^{m}}.
\]
Let us write $ \alpha = u(\widetilde{G}^{\circ l}(a)), \beta = u(\widetilde{G}^{\circ l}(b))$.
Since $\widetilde{G}^{\circ l}(a), \widetilde{G}^{\circ l}(b) \in p\Z_{p}$ and $u \in  L\lb z \rb$,
we have $ \alpha, \beta \in L$.

Now, if one of $ \alpha$ or $\beta$ is $0$, then so is the other.
In this case, $\widetilde{G}^{\circ l}(a) = \widetilde{G}^{\circ l}(b)$ and we are done.

Suppose $ \alpha$ and $ \beta$ are not $0$.
Then $\zeta = \alpha/ \beta$ is a $d^{m}$-th root of unity contained in $L$.
By Fact \ref{fact:rootofunity} (\ref{fact:rootofunity::3}) with $D=[L:\Q]$, there exists $\ell_0$ depending only on $L$ such that $\zeta^{\ell_0}=1$. If the prime factorization of $\ell_0$ is given by $\prod p_i^{e_i}$, then let $\ell'=\max\{e_i\mid p_i\textrm{\ divides\ }d\}$. It follows that $\ell'$ depends only on $L$ and $d$, and $\zeta^{d^{l'}}=1$. This means 
\[
(u(\widetilde{G}^{\circ l}(a)))^{d^{l'}} =(u(\widetilde{G}^{\circ l}(b)))^{d^{l'}},
\]
or equivalently
\[
\widetilde{G}^{\circ (l+l')} (a) = \widetilde{G}^{\circ (l+l')} (b).\qedhere
\]
\end{proof}

\section{Genus estimates}
\label{sec:genus}
For a curve $X$ over an algebraically closed field, the geometric genus of the smooth projective model of $X$ is denoted by $g_{X}$.

The following Proposition plays a key role in \S\ref{sec:sgcancellation}, specifically in Lemma \ref{lem:finitenessofcurves}. There it is used to show that a particular infinite chain of curves $C_i$ obtained by taking irreducible components of preimages of the diagonal $\Delta\subset\P^1\times\P^1$ must all have $g_{C_i}=0$.

\begin{prop}\label{prop:irr-genus}
Let $k$ be an algebraically closed field of characteristic zero.
Let $C, D$ be projective curves over $k$, and suppose $C$ is smooth.
Let $f \colon C \longrightarrow \P^{1}_{k}$ and $g \colon D \longrightarrow \P^{1}_{k}$
be surjective morphisms with $\deg f = n \geq 2$ and $\deg g = d$.
Let $X = (C \times_{\P^{1}} D)_{\rm red}$ be the reduced scheme of the fiber product:
 \[
\xymatrix{
X \ar[r] \ar[d] & D \ar[d]^{g}\\
C \ar[r]_{f} & \P^{1}_{k}.
}
\]

Assume that
\begin{enumerate}
\item $k(\P^{1}) \subset k(C)$ has no non-trivial intermediate field extension;
\item $\# f^{-1}(\infty) = 1$;
\item $ \widetilde{D} \stackrel{\nu}{\longrightarrow} D \longrightarrow \P^{1}$ is \'etale over $\infty$, where $\nu$ is the normalization map.
\end{enumerate}

Then $X$ is irreducible and 
\begin{align*}
2g_{X} -2 \geq n(2g_{D} -2 ) + (n-1)d.
\end{align*}
\end{prop}

\begin{rmk}
We can consider the following condition instead of (2):
$\infty$ is a branched point of $f$.
In this case, $X$ is still irreducible and the genus inequality becomes
\begin{align*}
2g_{X} -2 \geq n(2g_{D} -2 ) + d.
\end{align*}
\end{rmk}

\begin{proof}
Since $(C \times_{\P^{1}} \widetilde{D})_{\rm red}\to(C \times_{\P^{1}} D)_{\rm red}$ is surjective and an isomorphism over an open subset of the target, it suffices to replace $D$ by $\widetilde{D}$, and hence we may assume $D$ is smooth.\\
\\
\noindent{\bf Irreducibility of $X$.} Let $D' \longrightarrow D \longrightarrow \P^{1}$ be the Galois closure of $g$, i.e.
$D'$ is a smooth projective curve such that $k(D') \supset k(\P^{1})$ is the Galois closure of $k(D) \supset k(\P^{1})$. 
We write $g' \colon D' \longrightarrow \P^{1}$.
Since $g$ is \'etale over $\infty$, $g'$ is also \'etale over $\infty$ (cf.\ for example \cite[Proposition 3.2.10]{fu}).
Replacing $D$ with $D'$, we may assume $g \colon D \longrightarrow \P^{1}$ is Galois.
Let $X_{0} \subset X$ be an irreducible component.
Since $f, g$ are flat, $X_{0}$ dominates $C$ and $D$ and
we have inclusions $k(C), k(D) \subset k(X_{0})$:
 \[
\xymatrix{
 & k(X_{0}) \ar@{-}[dl] \ar@{-}[dr]&\\
 k(C) \ar@{-}[dr] & & k(D) \ar@{-}[dl]\\
 & k(\P^{1})&.
}
\]

By assumption (1), we have either 
\[
\text{$k(C) \cap k(D) = k(\P^{1})$ or $k(C) \cap k(D) = k(C)$}.
\]

First suppose $k(C) \cap k(D) = k(\P^{1})$.
Then since $k(D) \supset k(\P^{1})$ is a Galois extension, $k(C)$ and $k(D)$ are linearly disjoint over $k(\P^{1})$. 
Thus $[k(X_{0}) : k(D)] \geq [k(C):k(\P^{1})] = n$ and this implies $X_{0}= X$.

Next suppose $k(C) \cap k(D) = k(C)$.
In this case $k(D) \supset k(C)$ and there is $p \colon D \longrightarrow C$ such that
 \[
\xymatrix{
C \ar[rd]_{f} & & D \ar[ll]_{p} \ar[ld]^{g}\\
&\P^{1}&
}
\]
commutes.
Then we have
\begin{align*}
g^{*}[\infty] = p^{*}f^{*}[\infty] = np^{*}[c]
\end{align*}
 as divisors where $c \in C$ is the point such that $f(c) = \infty$.
 Since $g$ is \'etale over $\infty$, this is a contradiction.\\
\\
\noindent{\bf Genus inequality.} Finally we prove the inequality.
Let $ \widetilde{X} \longrightarrow X$ be the normalization and consider the following diagram:
 \[
\xymatrix@R=2mm{
\widetilde{X} \ar[rdd]^{\pi} \ar[dd] \ar@/_15pt/[dddd]_{q} & \\
&&\\
X \ar[r] \ar[dd] & D \ar[dd]^{g}\\
&&\\
C \ar[r]^{f} & \P^{1}\\
c \ar@{|->}[r] & \infty.
}
\]
By Riemann-Hurwitz, we have
\begin{align*}
2g_{X} - 2  = (\deg \pi)(2g_{D} -2) + \deg R_{\pi}
\end{align*}
where $R_{\pi}$ is the ramification divisor of $\pi$.
We have $\deg \pi = n$ since $X$ is irreducible.
Also, we have
\begin{align*}
\pi^{*}g^{*}[\infty] = q^{*}f^{*}[\infty] = nq^{*}[c].
\end{align*}
Since $g^{*}[\infty]$ is a reduced divisor of degree $d$, we get 
$\deg R_{\pi} \geq (n-1)d$ and we are done.
\end{proof}

The next lemma will be used in \S \ref{sec:sgcancellation} to confirm the assumptions in Proposition \ref{prop:irr-genus}
are satisfied.

\begin{lem}\label{lem:diagetale}
Let $k$ be an algebraically closed field of characteristic zero.
Let $P(x) \in k[x]$ be a polynomial of degree $n \geq 2$.
Let $X \subset \P^{1} \times \P^{1}$ be an irreducible component of $(P \times P)^{-1}( \Delta)$.
{\rm (}Here we consider $P$ as a self-morphism on $\P^{1}$.{\rm)}
Let $ \widetilde{X} \longrightarrow X$ be the normalization.
Then
\begin{align*}
\widetilde{X} \longrightarrow X \subset \P^{1} \times \P^{1} \xrightarrow{\pr_{i}} \P^{1}
\end{align*}
is \'etale over $\infty$, where $\pr_{i}$ is the $i$-th projection for $i=1,2$.
\end{lem}

\begin{proof}
Let us write 
\begin{align*}
&f \colon \widetilde{X} \longrightarrow X \subset \P^{1} \times \P^{1}  \xrightarrow{\pr_{1}} \P^{1} \\
&g \colon \widetilde{X} \longrightarrow X \subset \P^{1} \times \P^{1}  \xrightarrow{\pr_{2}} \P^{1}.
\end{align*}
Then by the following commutative diagram
 \[
\xymatrix@R=7mm{
\widetilde{X} \ar[r] & X \ar@{}[r]|-*[@]{\subset}  & (P\times P)^{-1}( \Delta) \ar[r] \ar@{}[d]|{\asubset}  & \Delta \ar@{}[d]|{\asubset} \ar@/^25pt/[dd]^{\id} \\
& & \P^{1} \times \P^{1} \ar[r]^{P \times P}\ar[d]^-{\pr_i} & \P^{1} \times \P^{1}\ar[d]^-{\pr_i}\\
& & \P^1\ar[r]^{P} & \P^1
}
\]
we see $P \circ f = P \circ g$ is a finite map. Since $P$ is finite, it follows that $f$ and $g$ are finite surjective. 
If we identify $f$ and $g$ with the pull-back of $x \in k(x) = k(\P^{1})$ under the maps $f^*,g^*\colon k(\P^{1})\to k( \widetilde{X}) = k(X)$, we see $k(X)$ is generated by $f$ and $g$. Also, $P \circ f = P \circ g$ is equivalent to $P(f) = P(g)$ as elements in $k( \widetilde{X})$.

Now, suppose there is a point $ \alpha \in \widetilde{X}$ such that
$f( \alpha) = \infty$ and $e_{f, \alpha} = d \geq 2$, where $e_{f, \alpha}$ is the ramification index of $f$ at $ \alpha$.
Since $P \circ f = P \circ g$ and $P$ is a polynomial map, we have $g( \alpha) = \infty$.
Moreover, we have $e_{g, \alpha} = d$, since 
\[
e_{g, \alpha} = \frac{e_{P\circ g, \alpha}}{e_{P, \infty}} = \frac{e_{P\circ f, \alpha}}{e_{P, \infty}} = e_{f, \alpha} = d.
\]

Let $t \in \O_{ \widetilde{X}, \alpha}$ be a uniformizer and write 
\[
t = \frac{H(f, g)}{L(f, g)}
\]
for some bivariate polynomials $H, L$ over $k$. Let us embed $\O_{ \widetilde{X}, \alpha}$ and $k( \widetilde{X})$ into 
the completion $ \widehat{\O}_{ \widetilde{ X}, \alpha} \simeq k\lb t \rb$
and its fraction field $k\lp t \rp$:
 \[
\xymatrix{
k( \widetilde{X}) \ar[rr] \ar@{}[d]|\usubset && k\lp t \rp \ar@{}[d]|\usubset \\
\O_{ \widetilde{X}, \alpha} \ar[r] & \widehat{\O}_{ \widetilde{ X}, \alpha} \ar@{}[r]|\simeq &  k\lb t \rb.
}
\]
The discrete valuation on $k\lp t \rp$ is denoted by $v$.
Since $f(\alpha)=\infty$ and $e_{f,\alpha}=d$, we have $v(f) = -d$ and hence $f$ is of the form
\[
f = \sum_{j \geq -d} c_{j}t^{j}
\]
where $c_{-d} \neq 0$.
Since $k$ is algebraically closed and has characteristic zero, by Hensel's lemma, any invertible element of $k\lb t \rb$
has a $d$-th root.
Hence, we can write 
\[
f = \frac{1}{u^{d}}
\]
for some $u \in k\lb t \rb$ with $v(u) = 1$. Since 
\[
k\lb u \rb \longrightarrow k\lb t \rb, u \mapsto u
\]
is an isomorphism, 
both $f=1/u^{d}$ and $g = g(u)$ are Laurent series of pole order $d$.

We claim that $g(u)$ is a Laurent series of $u^{d}$. Upon showing so, we have 
\begin{align*}
t = \frac{H(1/u^{d}, g(u))}{L(1/u^{d}, g(u))}
\end{align*}
in $k\lp u \rp$. Since the right-hand side has valuation divisible by $d$, we have arrived at a contradiction.

It remains to prove our claim. Suppose $g(u)$ involves terms $u^{j}$ with $d \nmid j$. Then we can write
\[
g(u) = g_{0}(u^{d}) + h(u)
\]
where $g_{0}, h \in k\lp u \rp$ such that they are of the form
\begin{align*}
&g_{0}(u) = \frac{a_{-1}}{u} + a_{0} + a_{1}u + \cdots, \quad a_{-1} \neq 0\\
&h(u) =  b_{i} u^{i} + b_{i+1} u^{i+1} + \cdots, \quad \text{$b_{i} \neq 0$, $i > -d$, and $d \nmid i$.}
\end{align*}
Recall that $P(f) = P(g)$.
Considering this equation in $k\lp u \rp$, we have
\begin{equation}
\label{eqn:Taylorexpansion}
P\left(\frac{1}{u^{d}}\right) = P(g_{0}(u^{d}) + h(u))=\sum_{m=0}^n P^{(m)}(g_{0}(u^{d}))\frac{h(u)^{m}}{m!}.
\end{equation}
Notice that
\begin{itemize}
\item $P(g_{0}(u^{d}))$ is a Laurent series of $u^{d}$;
\item the lowest degree term of $P'(g_{0}(u^{d}))h(u)$ has degree $-(n-1)d + i$;
\item for $2 \leq m \leq n$, the lowest degree term of $P^{(m)}(g_{0}(u^{d}))h(u)^{m}/m!$ has degree $-(n-m)d + mi$,
which is strictly larger than $-(n-1)d + i$.
\end{itemize}
Therefore, the right-hand side of \eqref{eqn:Taylorexpansion} has a $u^{-(n-1)d+i}$ term, while the left-hand side does not which is a contradiction.
\end{proof}

\section{Cancellation for several polynomial maps on $\P^{1}$}\label{sec:sgcancellation}

The goal of this section is to prove Theorem \ref{thm:cancellationsemigroup}. 
One of the key tools of the proof is the canonical height of a subvariety.
We summarize the basic properties of such heights in the next subsection.
Standard references for canonical heights of subvarieties are \cite{gubler-thesis,hut}.

\subsection{Canonical heights of subvarieties}
Throughout this subsection, we fix the following:

\begin{itemize}
\item $X$ is a (geometrically irreducible) projective variety over a number field $K$.
\item $H$ is an ample divisor on $X$
\item $f_{1}, \dots , f_{r} \colon X \longrightarrow X$ are surjective morphisms such that $f_{i}^{*}H \sim d_{i} H$ for some $d_{i}>1$.
\end{itemize}

Replacing $H$ by a multiple, we can assume the embedding $i_{H} \colon X \longrightarrow \P^{N}$ defined by $H$ is equivariant with respect to
all $f_{i}$, i.e.~there are surjective morphisms $F_{i} \colon \P^{N}_{K} \longrightarrow \P^{N}_{K}$ such that 
 \[
\xymatrix{
 \P^{N} \ar[r]^{F_{i}} & \P^{N}\\
 X \ar[u]^{i_{H}} \ar[r]_{f_{i}} & X \ar[u]_{i_{H}}
}
\]
commutes for all $i$.

Using this embedding, one may define canonical heights $\hat{h}_{f_{i}}$ for subvarieties of $X$; one may also define the height $h$ for subvarieties via this embedding. The degree of a subvariety of $X$ is its degree when considered as a subvariety of $\P^{N}$, i.e.~the degree with respect to $H$. The main facts we use are summarized as follows.

\begin{fact}[cf. \cite{hut}]\ 
\label{fact:hts-of-subvarieties}

\begin{enumerate}
\item
We have
\begin{align*}
\hat{h}_{f_{i}}(f_{i}(Z)) = \frac{d_{i} \deg f_{i}(Z)}{\deg Z} \hat{h}_{f_{i}}(Z)
\end{align*}
for subvarieties $Z \subset X$.

\item
Given a non-negative integer $D$, there is a positive constant $C=C(D,f_1,\dots,f_r)$ such that for all subvarieties $Z\subset X$ of degree $D$, we have
\begin{align*}
|\hat{h}_{f_{i}}(Z) - h(Z)| \leq C.
\end{align*}

\item\label{fact:hts-of-subvarieties::boundedness}
There are only finitely many subvarieties over $K$ with bounded degree and bounded height $h$.

\end{enumerate}
\end{fact}

\begin{prop}\label{prop:sght}
Let $D$ be a positive integer.

Suppose we have a sequence 
 \[
\xymatrix{
\cdots \ar[r] & Y_{j} \ar[r]  \ar@{}[d]|{\rotatebox[origin=c]{-90}{$\subset$}} & Y_{j-1} \ar[r]  \ar@{}[d]|{\rotatebox[origin=c]{-90}{$\subset$}}
 & \cdots \ar[r]  & Y_{1} \ar[r]  \ar@{}[d]|{\rotatebox[origin=c]{-90}{$\subset$}} & Y_{0} \ar@{}[d]|{\rotatebox[origin=c]{-90}{$\subset$}}\\
\cdots \ar[r]  & X \ar[r]_{f_{i_{j}}} & X \ar[r]_{f_{i_{j-1}}} & \cdots \ar[r] & X \ar[r]_{f_{i_{1}}}  & X
}
\]
where $Y_{j}$ are irreducible subvarieties over $K$ of dimension $k$.

If $\deg Y_{j} \leq D$ for all $j$, then there is a constant $C=C(D)$ depending only on $D$ (and not the sequence) such that
\begin{align*}
h(Y_{j}) \leq a \frac{\deg (f_{i_{j}}|_{Y_{j}})}{d_{i_{j}}^{k+1}} h(Y_{j-1}),
\end{align*}
where $\deg (f_{i_{j}}|_{Y_{j}})$ is the topological degree of the restriction of $f_{i_{j}}$ to $Y_{j}$ and 
\begin{align*}
a = \left(1 - \frac{C}{h(Y_{j})} - \frac{C \deg Y_{j}}{d_{i_{j}} \deg Y_{j-1} h(Y_{j})  }\right)^{-1},
\end{align*}
provided that $h(Y_{j})$ and the denominator of $a$ are not zero.

\end{prop}

\begin{proof}
Take $C>0$ such that
\begin{align*}
|\hat{h}_{f_{i}}(Z) - h(Z)| \leq C
\end{align*}
holds for all $i=1,\dots , r$ and all subvarieties $Z \subset X$ with $\deg Z \leq D$. Then 
\begin{align*}
h(Y_{j-1}) = h(f_{i_{j}}(Y_{j})) &\geq \hat{h}_{f_{i_{j}}}(f_{i_{j}}(Y_{j}))-C \\
& = \frac{d_{i_{j}} \deg Y_{j-1}}{\deg Y_{j}} \hat{h}_{f_{i_{j}}}(Y_{j}) -C\\
& \geq \frac{d_{i_{j}} \deg Y_{j-1}}{\deg Y_{j}} ( h(Y_{j}) - C ) -C.
\end{align*}
By simplifying this expression, we have
\begin{align*}
h(Y_{j}) \leq a \frac{\deg Y_{j}}{ d_{i_{j}} \deg Y_{j-1}} h(Y_{j-1}).
\end{align*}
By the projection formula, we see
\begin{align*}
\deg Y_{j-1}& = (H^{k} \cdot [Y_{j-1}]) = (H^{k} \cdot [f_{i_{j}}(Y_{j})]) = \frac{1}{ \deg (f_{i_{j}}|_{Y_{j}} )}(H^{k} \cdot (f_{i_{j}})_{*}[Y_{j}])\\
& = \frac{1}{ \deg (f_{i_{j}}|_{Y_{j}} )}(f_{i_{j}}^{*}H^{k} \cdot [Y_{j}]) = \frac{d_{i_{j}}^{k}}{ \deg (f_{i_{j}}|_{Y_{j}} )}(H^{k} \cdot [Y_{j}]) \\
&= \frac{d_{i_{j}}^{k}}{ \deg (f_{i_{j}}|_{Y_{j}} )} \deg Y_{j}.
\end{align*}
Thus we have
\begin{align*}
\frac{\deg Y_{j}}{ d_{i_{j}} \deg Y_{j-1}} = \frac{\deg (f_{i_{j}}|_{Y_{j}} )}{d_{i_{j}}^{k+1}},
\end{align*}
finishing the proof.
\end{proof}

\begin{rmk}
\label{rmk:aleqaM}
If $Y_{j}$ are codimension one, then 
$\frac{\deg (f_{i_{j}}|_{Y_{j}} )}{d_{i_{j}}^{k+1}} \leq 1$
since $d_{i_{j}}^{k+1} = d_{i_{j}}^{\dim X}$ is the topological degree of $f_{i_{j}}$. If we additionally assume $k=1$, then $\frac{\deg(Y_j)}{d_{i_j}\deg(Y_{j-1})}=\frac{\deg (f_{i_{j}}|_{Y_{j}} )}{d_{i_{j}}^{k+1}}\leq1$, and so $a\le \left(1-\frac{2C}{h(Y_j)}\right)^{-1}$.
\end{rmk}

\subsection{Proof of Theorem \ref{thm:cancellationsemigroup}}

We begin by reducing Theorem \ref{thm:cancellationsemigroup} to the case of simpler polynomial maps.

\begin{defn}
Let $\f \colon \P^{1}_{k} \longrightarrow \P^{1}_{k}$ be a surjective self-morphism over an algebraically closed field $k$.
We say $f$ is \emph{indecomposable} if the function field extension induced by $f$ does not have any non-trivial intermediate field extensions.
\end{defn}

Note that every surjective self-morphism can be written as the composition of indecomposable self-morphisms, although this decomposition is not unique in general.

\begin{lem}\label{lem:redtoindec}
To prove Theorem \ref{thm:cancellationsemigroup}, we may assume the base change $(\f_{i})_{ \overline{K}}$ are indecomposable. 
\end{lem}
\begin{proof}
Let us write $(\f_{i})_{ \overline{K}} = \psi_{i1}\circ \cdots \circ \psi_{it_{i}}$,
where $\psi_{ij} \colon \P^{1}_{ \overline{K}} \longrightarrow \P^{1}_{ \overline{K}}$ are indecomposable morphisms of degree at least two.
Since $(\f_{i})_{ \overline{K}}$ is a polynomial map, we may take all $\psi_{ij}$ as polynomial maps.
Now there is a finite extension $L$ of $K$ over which all $\psi_{ij}$ are defined.
Then Theorem \ref{thm:cancellationsemigroup} for the polynomial maps $\psi_{ij} \colon \P^{1}_{L} \longrightarrow \P^{1}_{L}$ implies Theorem \ref{thm:cancellationsemigroup}
for $\f_{i}$.
\end{proof}

Now let
\begin{align}
\notag\bullet \ &X = \P^{1}_{K} \times_{K} \P^{1}_{K}\\[2mm]
\label{notation:maps}\bullet \ &f_{i} = \f_{i} \times \f_{i} \text{ ,where $\f_{1}, \dots , \f_{r}$ are } \\
\notag&\text{polynomial maps on $\P^{1}_{K}$ indecomposable over $ \overline{K}$}\\[2mm]
\notag\bullet \ &\deg \f_{i}=d_{i} \geq 2 
\end{align}

Take $\d \in \Z_{>0}$ so that for all $i$, there exists $F_i$ so that the diagram
 \[
\xymatrix{
 \P^{N} \ar[r]^{F_{i}} & \P^{N}\\
\P^{1} \times \P^{1}  \ar[u]^{i_{\O(\d,\d)}} \ar[r]_{f_{i}} & \P^{1} \times \P^{1} \ar[u]_{i_{\O(\d,\d)}}
}
\]
commutes. Define the canonical height $\hat{h}_{f_{i}}$ and height $h$ for subvarieties by this embedding. 
Let $D = 4\d$ and let $C = C(D)$ be the constant as in Proposition \ref{prop:sght}.
(Note that the degree of subvarieties is with respect to $\O(\d,\d)$.)
 
 Fix $M>0$ large enough so that
\begin{itemize}
\item $h( \Delta) < M$
\item $a(M)^{2}B < 1$
\end{itemize}
 where $ \Delta \subset \P^{1} \times \P^{1}$ is the diagonal and
 \begin{align*}
 &a(M) := \frac{1}{1- \frac{2C}{M}}\\
 &B := \max\left\{ \frac{d_{i}^{2}-1}{d_{i}^{2}} \right\}.
 \end{align*}
  Let 
 \begin{align*}
 \Sigma = \{ Z \subset \P^{1}_{K} \times \P^{1}_{K} \mid \deg_{\O(\d,\d)}Z \leq 4\d,\, h(Z) < a(M)M \}
 \end{align*}
 be the set of irreducible closed subschemes of bounded degree and bounded height.
 This is a finite set depending on $\f_{1}, \dots , \f_{r}$ by Fact \ref{fact:hts-of-subvarieties} (\ref{fact:hts-of-subvarieties::boundedness}). The following is the key lemma used to prove Theorem \ref{thm:cancellationsemigroup}.

 \begin{lem}\label{lem:finitenessofcurves}
 Suppose we have a sequence 
 \[
\xymatrix{
\cdots \ar[r] & C_{j} \ar[r]  \ar@{}[d]|{\rotatebox[origin=c]{-90}{$\subset$}} & C_{j-1} \ar[r]  \ar@{}[d]|{\rotatebox[origin=c]{-90}{$\subset$}}
 & \cdots \ar[r]  & C_{1} \ar[r]  \ar@{}[d]|{\rotatebox[origin=c]{-90}{$\subset$}} & C_{0} \ar@{}[d]|{\rotatebox[origin=c]{-90}{$\subset$}}\\
\cdots \ar[r]  & \P^{1}\times \P^{1} \ar[r]_-{f_{i_{j}}} & \P^{1}\times \P^{1} \ar[r]_-{f_{i_{j-1}}} & \cdots \ar[r] & \P^{1}\times \P^{1} \ar[r]_-{f_{i_{1}}}  & \P^{1}\times \P^{1}
}
\]
such that $C_{0}= \Delta$ and $C_{j}$ are irreducible curves and $C_{j}(K)$ is Zariski dense in $C_{j}$ for all $j$.
Then we have $C_{j} \in \Sigma$.
 \end{lem}
 
\begin{proof}
First note that the genus of each $C_j$ is $g_{C_{j}}=0, 1$.
Let us consider the diagram
 \[
\xymatrix{
C_{j+1} \ar[r]^{h_{j+1}} \ar@{}[d]|{\rotatebox[origin=c]{-90}{$\subset$}} & C_{j} \ar@{}[d]|{\rotatebox[origin=c]{-90}{$\subset$}} \ar@/^30pt/[dd]^{\nu_{j,t}}\\
\P^{1}\times \P^{1} \ar[r]^{f_{j+1}} \ar[d]_{\pr_{t}} & \P^{1}\times \P^{1} \ar[d]^{\pr_{t}} \\
\P^{1} \ar[r]_{\f_{i_{j+1}}} & \P^{1}
}
\]
where $t$ is $1$ or $2$.

Since $C_{j}$ is an irreducible component of 
\[
(f_{i_{1}}\circ \cdots \circ  f_{i_{j}})^{-1}( \Delta)
\]
and each $f_{i_{l}}$ is the self-product of a polynomial map,
Lemma \ref{lem:diagetale} says that the composite of the normalization of $C_{j}$ and $\nu_{j,t}$
is \'etale over $\infty$.
Hence we can apply Proposition \ref{prop:irr-genus} to $\f_{i_{j+1}}$ and $\nu_{j,t}$
and get
\[
2g_{C_{j+1}}-2 \geq d_{i_{j+1}}(2 g_{C_{j}}-2) + (d_{i_{j+1}}-1) \deg \nu_{j,t}.
\]

If $g_{C_{j}} =1$, then $2g_{C_{j+1}}-2 > 0$ which is a contradiction.
Thus $g_{C_{j}}=0$ for all $j$ and we have
\[
-2 \geq -2 d_{i_{j+1}} + (d_{i_{j+1}} - 1)\deg \nu_{j,t}.
\]
If $\deg \nu_{j, t} \geq 3$, then $-2 \geq d_{i_{j+1}} -3 \geq -1$ which is a contradiction.

Thus we have 
\begin{equation}\label{eqn:nujt-leq-2}
\deg \nu_{j,t} \leq 2
\end{equation}
for all $j$ and $t$.
This implies $\deg_{\O(\d,\d)} C_{j} \leq 4\d$.

Fix $j$. It remains to prove $h(C_j)<a(M)M$. We may assume $h(C_j)\geq M$ since otherwise $h(C_j)<M<a(M)M$. Recall $h(C_0)=h(\Delta)<M$, and hence $j\geq1$. For ease of notation, let
\[
H(\ell):=\frac{\deg h_{\ell}}{d_{i_{\ell}}^{2}}.
\]
Then by Proposition \ref{prop:sght} and Remark \ref{rmk:aleqaM}, we have
\[
h(C_j)\leq a(M)H(j)h(C_{j-1}).
\]
If $h(C_{j-1})<M$, then we see $h(C_j)<a(M)M$. So, we may assume $h(C_{j-1})\geq M$, in which case $j\geq2$. Another application of Proposition \ref{prop:sght} and Remark \ref{rmk:aleqaM} shows
\[
h(C_j)\leq a(M)^2H(j)H(j-1)h(C_{j-2}).
\]
We claim that
\begin{equation}\label{eqn:HjHj-1leqB}
H(j)H(j-1)\leq B
\end{equation}
and hence
\[
h(C_j)\leq a(M)^2 B h(C_{j-2})<h(C_{j-2});
\]
it then follows from induction that $h(C_j)<a(M)M$.

It therefore remains to prove \eqref{eqn:HjHj-1leqB}. It suffices to prove that for all $\ell\geq0$,
\begin{equation}\label{eqn:HjHj-1leqB::1}
H(\ell)<1\ \Rightarrow\ H(\ell)\leq B
\end{equation}
and for all $\ell\geq2$,
\begin{equation}\label{eqn:HjHj-1leqB::2}
H(\ell)=1\ \Rightarrow\ H(\ell-1)<1.
\end{equation}
Indeed, \eqref{eqn:HjHj-1leqB::1} and \eqref{eqn:HjHj-1leqB::2} imply that at least one of $H(j)$ and $H(j-1)$ is strictly less than 1, and hence at most $B$; since $H(\ell)\leq1$ for all $\ell$, we then see $H(j)H(j-1)\leq B$.

The proof of \eqref{eqn:HjHj-1leqB::1} is straightforward. Since $\deg(h_\ell)$ is an integer, if $H(\ell)<1$, then $H(\ell)\leq\frac{d_{i_\ell}^2-1}{d_{i_\ell}^2}\leq B$.

To prove \eqref{eqn:HjHj-1leqB::2}, assume $\ell\geq1$ and consider the following diagram.
 \[
\xymatrix{
C_{\ell} \ar@{}[d]|{\rotatebox[origin=c]{-90}{$\subset$}} \ar[r]^{h_{\ell}} \ar@/_25pt/[dd]_{\nu_{\ell,t}} &
C_{\ell-1} \ar@{}[d]|{\rotatebox[origin=c]{-90}{$\subset$}} \ar@/^25pt/[dd]^{\nu_{\ell-1,t}} \\
\P^{1}\times \P^{1} \ar[d]^{\pr_{t}} \ar[r]^{f_{\ell}} &  \P^{1}\times \P^{1} \ar[d]^{\pr_{t}}\\
\P^{1} \ar[r]_{\f_{i_{\ell}}} & \P^{1}
}
\]
We see
\begin{align*}
\deg \nu_{\ell,t} = \frac{\deg h_{\ell} \deg \nu_{\ell-1,t}}{d_{i_{\ell}}} = H(\ell)d_{i_\ell}\deg\nu_{\ell-1,t}.
\end{align*} 
Combining \eqref{eqn:nujt-leq-2} with the fact that $d_{i_\ell}\geq2$, we see if $H(\ell)=1$, then $\deg \nu_{\ell,t}=d_{i_\ell}=2$ and $\deg \nu_{\ell-1,t}=1$. In particular, if $\ell\geq2$, then we cannot have $H(\ell)=H(\ell-1)=1$ since this would imply $1=\deg \nu_{\ell-1,t}=2$.
\end{proof}

\begin{proof}[Proof of Theorem \ref{thm:cancellationsemigroup}]
By Lemma \ref{lem:redtoindec}, we may assume $\f_{i}$ are indecomposable over $  \overline{K}$.
We use the notation as in (\ref{notation:maps}).

Let us inductively define sets $\Sigma_{0}, \Sigma_{1}, \dots$ of integral curves on $\P^{1}_{K} \times_{K} \P^{1}_{K}$ as follows:
\begin{itemize}
\item $\Sigma_{0}=\{ \Delta\}$;
\item having defined $\Sigma_{0}, \dots , \Sigma_{m}$, define
\begin{align*}
\Sigma_{m+1} = \left\{ C \subset \P^{1}_{K} \times_{K} \P^{1}_{K} \middle| \txt{ $C \notin \bigcup_{i=0}^{m} \Sigma_{i}, \# C(K)= \infty$,\\
 $f_{j}(C) \in \Sigma_{m}$ for some $j$} \right\}.
\end{align*}
\end{itemize}
Note that all $\Sigma_{i}$ are finite.

We claim that $\Sigma_{n} =  \emptyset$ for all large $n$. 
Indeed, suppose $\Sigma_{n} \neq  \emptyset$ for all $n \geq 0$.
Then there exists an infinite sequence
\begin{align*}
\cdots \xrightarrow{f_{i_{n+1}}} C_{n} \xrightarrow{f_{i_{n}}} C_{n-1} \xrightarrow{f_{i_{n-1}}} \cdots \xrightarrow{f_{i_{2}}} C_{1} \xrightarrow{f_{i_{1}}} \Delta
\end{align*}
such that $C_{i} \in \Sigma_{i}$. By Lemma \ref{lem:finitenessofcurves}, every $C_i\in\Sigma$. This implies $\Sigma$ is an infinite set, contradicting Fact \ref{fact:hts-of-subvarieties} (\ref{fact:hts-of-subvarieties::boundedness}).

Now let $\Omega = \bigcup_{n\geq 0}\Sigma_{n}$. This is a finite set of curves as we just proved.
Let $Z_{0} = \bigcup_{C \in \Omega}C$ and define
\begin{align*}
T = \{ x \in (\P^{1} \times \P^{1})(K) \mid \text{$x \notin Z_{0}$, $f_{j}(x) \in Z_{0}$ for some $j$}  \}.
\end{align*}
Then $T$ is finite.
Indeed, for any $C\in \Omega$, every irreducible component of $f_{j}^{-1}(C)$ is either an element of $\Omega$ or has finitely many $K$-points.

Let $\hat{h}_{f_{1}},\dots , \hat{h}_{f_{r}}$ be canonical heights on $(\P^{1} \times \P^{1})( \overline{K})$ of $f_{1},\dots, f_{r}$ with respect to $\O(1,1)$.
Let $h$ be a Weil height on $(\P^{1} \times \P^{1})( \overline{K})$ with respect to $\O(1,1)$.
Pick constants $B_{0}, B_{1} > 0$ so that for all $i$,
\begin{align*}
&|\hat{h}_{f_{i}} - h| < B_{0}\\
&\hat{h}_{f_{i}}(x) < B_{1} \quad \text{for all $x \in T$.}
\end{align*}
Set 
\begin{align*}
S = \{ y \in (\P^{1} \times \P^{1})(K) \mid h(y) < 3B_{0} + B_{1} \}.
\end{align*}
For every $x\in T$, we have $h(x)<\hat{h}_{f_1}(x)+B_0<B_0+B_1$, and so $T \subset S$.

Set $Z = S \cup Z_{0}$.
We prove this $Z$ has the desired property.
To this end, it is enough to prove the following:
if $(a,b) \in (\P^{1} \times \P^{1})(K)$ satisfies $f_{j}(a,b) \in Z$ for some $j \in \{1, \dots, r\}$, then $(a,b) \in Z$.

Suppose $f_{j}(a,b) \in Z$.
If $f_{j}(a,b) \in S$, then
\begin{align*}
d_{j}\hat{h}_{f_{j}}(a,b) = \hat{h}_{f_{j}}(f_{j}(a,b)) < h(f_{j}(a,b)) + B_{0} < 4B_{0} + B_{1}.
\end{align*}
Since $d_{j} \geq 2$, we have
\begin{align*}
h(a,b) < \hat{h}_{f_{j}}(a,b) + B_{0} \leq \frac{1}{2}(4B_{0} + B_{1}) + B_0 < 3B_{0} + B_{1},
\end{align*}
and hence $(a,b) \in S \subset Z$.
If $f_{j}(a,b) \in Z_{0}$, then by the definition of $T$, $(a,b) \in Z_{0}$ or $(a,b) \in T$ and we are done.
\end{proof}

\section{Counter-examples}\label{sec:counter-exs}

\subsection{Counter-example for non-projective varieties}
\label{subsec:counter-ex-non-proj}
We give an example demonstrating the necessity of the projective assumption in Conjecture \ref{cancelconj} (and hence that Question \ref{PIC} has a negative answer without this hypothesis). In our example, $X=\A^2$ and $Y=\{0\}$. Let $f\colon\A^2\to\A^2$ be given by
\[
f(x,y)=(xy,x^2+xy).
\]
Let $\{F_n\}$ be the sequence of Fibonacci numbers with $F_0=0$ and $F_1=1$, and let
\[
a_n = \begin{cases}
((-1)^{n+1} F_n, (-1)^n F_{n-1}), & n\geq1\\
(0,1), & n=0\\
(0,0), & n=-1
\end{cases}.
\]
Then for $n\geq0$, we have
\[
f(a_n)=(-1)^nF_na_{n-1}
\]
and hence
\[
f^{n+1}(a_n)=0\quad\textrm{and}\quad f^{n}(a_n)\neq0.
\]
As a result, for all $n\neq0$,
\[
f^{n+1}(a_n)=f^{n+1}(a_{n-1})\quad\textrm{and}\quad f^{n}(a_n)\neq f^n(a_{n-1}).
\]

\subsection{Counter-example over local fields}
\label{subsec:counter-ex-local-fields}
We show that Conjecture \ref{cancelconj} is not true over local fields, and hence one cannot, in general, employ $p$-adic methods to solve the conjecture.

\begin{prop}
\label{prop:counter-ex-local-fields}
Let $f\colon\P^2_\Q\to\P^2_\Q$ be given by
\[
f(x:y:z)=(x^{2} +yz : x^{2} + y^{2} + yz : z^2).
\]
Then for all $p\leq\infty$, the analogue of Conjecture \ref{cancelconj} fails for the base change $f_{\Q_p}:=f\times_\Q\Q_p$, i.e., for all $p\leq\infty$ and all $n\geq1$, there exist $a_{n,p},b_{n,p}\in\P^2(\Q_p)$ such that
\[
f^n_{\Q_p}(a_{n,p})=f^n_{\Q_p}(b_{n,p})\quad\textrm{and}\quad f^{n-1}_{\Q_p}(a_{n,p})\neq f^{n-1}_{\Q_p}(b_{n,p}).
\]
\end{prop}
\begin{proof}
Since $f$ is finite and $\P^2$ is smooth, $f$ is flat by \cite[Corollary 18.17]{eisenbud}. Letting $\A^2\subset\P^2$ be the locus where $z\neq0$, we see $f(\A^2)\subset\A^2$. Let
\[
g=(f|_{\A^2}\times f|_{\A^2})\colon\A^4\to\A^4.
\]
We denote the coordinates on $\A^4$ by $(x,y,x',y')$ and let $\Delta=V(x-x',y-y')\subset\A^4$ be the diagonal. It is enough to prove
\[
\left(g^{-(n+1)}( \Delta) \smallsetminus g^{-n}( \Delta) \right)(\Q_{p})
\]
is infinite for all $n \geq 1$ and $p \leq \infty$.

For ease of notation, given any polynomial $F(x,y)\in\Q[x,y]$, we let $F'\in\Q[x',y']$ be the polynomial $F':=F(x',y')$. Let us write $f|_{\A^2}^{n} = (P_{n}, Q_{n})$. Then
\begin{align*}
&P_{n+1} = P_{n}^{2} + Q_{n}\\
&Q_{n+1} = P_{n}^{2} + Q_{n}^{2} + Q_{n}.
\end{align*}
We see
\[
g^{-n}( \Delta) = V(P_n-P'_n,Q_n-Q'_n)
\]
and so
\begin{align*}
g^{-(n+1)}( \Delta) 
&=V\!\left(P_{n}^{2} + Q_{n}-(P'_{n})^{2} - Q'_{n},\ \, Q_{n}^{2}  - (Q'_{n})^{2}\right)\\
&\supset V\!\left(P_{n}+P'_{n},\ \, Q_{n}  - Q'_{n}\right)=:Z_{n+1}
\end{align*}
Since the origin of $\A^2$ is a fixed point of $f$, we see $P_n(0,0)=Q_n(0,0)=0$; it follows that the origin $0\in\A^4$ lies on $Z_{n+1}$.

Letting $\sigma$ be the automorphism of $\A^4$ defined by $\sigma(x,y,x',y')=(x,y,-x',y')$, we see $Z_{n+1}=(\sigma g^n)^{-1}(\Delta)$. Since $g$ is finite flat, $Z_{n+1}$ is pure of dimension $2$. In particular, there is a $2$-dimensional irreducible component $W_{n+1}$ of $Z_{n+1}$ containing $0$. Next, notice that $Z_{n+1} \cap g^{-n}(\Delta)=V(P_n,P'_n,Q_n-Q'_n)=g^{-n}(V(x,x',y-y'))$. Again by flatness of $g$, we see $Z_{n+1} \cap g^{-n}(\Delta)$ is pure of dimension $1$, and in particular, $W_{n+1} \not\subset  g^{-n}( \Delta) $.

We will show that $Z_{n+1}$ is smooth at $0$ and hence $W_{n+1}$ is the unique irreducible component of $Z_{n+1}$ passing through $0$. We see then that $0$ defines a smooth $\Q_p$-point of $W_{n+1}$, and so $(W_{n+1})(\Q_p)$ is dense in $W_{n+1}\times_\Q\Q_p$ by \cite[Proposition 3.5.75]{poonen-rat-pts}. In particular, $\left(g^{-(n+1)}( \Delta) \smallsetminus g^{-n}( \Delta) \right)(\Q_{p})$ is infinite.

To prove $Z_{n+1}$ is smooth at $0$, observe that by induction
\[
\frac{ \partial P_{n}}{ \partial x}(0)=\frac{ \partial Q_{n}}{ \partial x}(0) = 0\quad\textrm{and}\quad\frac{ \partial P_{n}}{ \partial y}(0)=\frac{ \partial Q_{n}}{ \partial y}(0) = 1
\]
for all $n\geq 1$. Setting 
\[
\Phi_n = P_{n}+P'_n\quad\textrm{and}\quad\Psi_n = Q_{n}-Q'_n,
\]
it follows that
\begin{align*}
\frac{ \partial (\Phi_n, \Psi_n)}{ \partial (x, y, x', y')}(0) = 
\begin{pmatrix}
0 & 1 & 0 & 1\\
0 & 1 & 0 & -1
\end{pmatrix}
\end{align*}
and so $Z_{n+1}$ is smooth at $0$.
\end{proof}

\end{document}